\documentclass[a4paper,12pt]{amsart}
\usepackage{amsmath,amssymb,amsthm,amsxtra,mathrsfs}
\usepackage{a4wide}

\newcommand{\C}{{\mathbb C}}
\newcommand{\N}{{\mathbb N}}
\newcommand{\R}{{\mathbb R}}
\newcommand{\Z}{{\mathbb Z}}

\newcommand{\abs}[2][\empty]{\ifx#1\empty\left|#2\right|%
\else#1\vert #2 #1\vert\fi}
\newcommand{\defstyle}[1]{#1}
\newcommand{\eps}{\varepsilon}
\newcommand{\Fourier}{\mathcal F}
\newcommand{\fourier}{\widehat}
\renewcommand{\implies}{\Rightarrow}
\newcommand{\norm}[2][\empty]{\ifx#1\empty\left\Vert#2\right\Vert%
\else#1\Vert #2 #1\Vert\fi}
\newcommand{\Sphere}{\mathbf S}
\newcommand{\Schwartz}{\mathscr{S}}
\newcommand{\test}{\mathcal D}

\newcommand{\Gen}{{\mathcal G}}
\newcommand{\GenC}{\widetilde\C}
\newcommand{\GenR}{\widetilde\R}
\newcommand{\EMod}{{\mathcal E}}
\newcommand{\Null}{{\mathcal N}}
\newcommand{\caninf}{\rho}
\newcommand{\Gentdinfty}{\mbox{\small \raisebox{-.2ex}{$\dot{\widetilde\Gen}$}}{}^\infty}
\newcommand{\slow}{\mathit{slow}}
\newcommand{\fast}{\mathit{fast}}

\newcommand{\monad}{\mu}

\newtheorem{thm}{Theorem}[section]
\newtheorem{lemma}[thm]{Lemma}
\newtheorem{prop}[thm]{Proposition}
\newtheorem{cor}[thm]{Corollary}

\theoremstyle{definition}
\newtheorem{df}[thm]{Definition}
\newtheorem{rem}[thm]{Remark}

\begin{document}
\title[Microlocal analysis based on generalized points and generalized directions]{Microlocal analysis in generalized function algebras based on generalized points and generalized directions}
\author{Hans Vernaeve}
\thanks{Supported by research grant 1.5.138.13N of the Research Foundation Flanders FWO}
\address{Ghent University, Krijgslaan 281\\B-9000 Gent, Belgium.}
\email{hvernaev@cage.ugent.be}

\date{}

\begin{abstract}
We develop a refined theory of microlocal analysis in the algebra $\Gen(\Omega)$ of Colombeau generalized functions. In our approach, the wave front is a set of generalized points in the cotangent bundle of $\Omega$, whereas in the theory developed so far, it is a set of nongeneralized points. We also show consistency between both approaches.
\end{abstract}
\maketitle

\section{Introduction}
Differential algebras of generalized functions containing the space of Schwartz distributions have been constructed starting with the work of J.F. Colombeau \cite{Colombeau84,Colombeau85}. The theory has found diverse applications in the study of partial differential equations \cite{Hormann-DeHoop,Hormann-Ober-Pilipovic,Ob92}, providing a framework in which nonlinear equations and equations with strongly singular data or coefficients can be solved and in which their regularity can be analyzed.

The natural extension of microlocal analysis of Schwartz distributions to the Colombeau generalized function algebras is the so-called $\Gen^\infty$-microlocal analysis, which has been developed using the concept of $\Gen^\infty$-regularity \cite{Dapic98}. In this setting, the $\Gen^\infty$-wave front set has been defined as a set of (nongeneralized) points in the cotangent bundle of the domain of the generalized function \cite{Hormann99,GarettoPhD}. On the other hand, generalized functions in these algebras can be viewed as pointwise functions on the so-called generalized points of their domain. Moreover, since equations with strongly singular data or coefficients in Colombeau algebras are modeled by regularization, the corresponding differential operators themselves become generalized operators \cite{Hormann-Ober-Pilipovic}. Hence it is to be expected that the most suitable setting to study the propagation of singularities under such operators is by means of generalized objects (generalized characteristic varieties, etc.). From this point of view, a natural refinement of $\Gen^\infty$-microlocal analysis is to define also the wave front set as a set of generalized points in the cotangent bundle of $\Omega$.

In this paper, we start the development of this refinement of $\Gen^\infty$-microlocal analysis, which we call $\widetilde\Gen^\infty$-microlocal analysis. This also allows us to look at smaller (generalized) neighbourhoods than in the classical theory. We were unable to develop the theory on the very fine scale of the so-called sharp neighbourhoods (both for points in the domain and for directions in the Fourier domain), but the theory works well if we confine ourselves to the so-called slow scale neighbourhoods (see below). We show that the projection of the wave front set in the first coordinate is the $\widetilde\Gen^\infty$-singular support (Theorem \ref{proj-of-WF-is-singular-support}) and we characterize $\Gen^\infty$-microlocal regularity (at nongeneralized points and in nongeneralized directions) in terms of $\widetilde\Gen^\infty$-microlocal regularity (Theorem \ref{consistency}).

\section{Preliminaries}
Let $\Omega\subseteq \R^d$ be open. We denote the (so-called special) Colombeau algebra on $\Omega$ by $\Gen(\Omega)$, the set of compactly supported points in $\Omega$ by $\widetilde\Omega_c$ and the ring of Colombeau generalized real (resp.\ complex) numbers by $\GenR$ (resp.\ $\GenC$). We denote by $[u_\eps]\in\Gen(\Omega)$ (resp.\ $[x_\eps]\in\GenR^d$) the element with representative $(u_\eps)_\eps$ (resp.\ $(x_\eps)_\eps$). Such $u=[u_\eps]\in\Gen(\Omega)$ be identified with the map $\widetilde\Omega_c\to\GenC$
\[u([x_\eps]):=[u_\eps(x_\eps)].\]
We refer to \cite[\S 1.2]{GKOS} for definitions and further properties. We further denote
\[\caninf:= [\eps]\in\GenR, \quad \Sphere := \{x\in\R^d: \abs x = 1\}, \quad \widetilde\Sphere := \{x\in\GenR^d: \abs x = 1\},\]
\[\R_{>0}:= \{x\in\R: x>0\},\quad \GenR_{>0}:= \{x\in\GenR: x\ge 0, x \text{ invertible}\},\]
\[[x_\eps]\approx 0\iff x_\eps\to 0\text{ as }\eps\to 0.\]
We call $x\in\GenR^d$ \defstyle{fast scale} if $x$ belongs to
\[\GenR^d_{fs}:= \{x\in\GenR^d: (\exists a\in\R_{>0}) (\abs x\ge \caninf^{-a})\}\]
and we call $x$ \defstyle{slow scale} if $x$ belongs to
\[\GenR^d_{ss}:=\{x\in\GenR^d: (\forall a\in\R_{>0}) (\abs x\le \caninf^{-a})\}.\]
Due to the fact that the order on $\GenR$ is not total, $\GenR^d_{fs}\cup \GenR^d_{ss}\subsetneq \GenR^d$.\\
We call $x\in\GenR^d$ a \defstyle{slow scale infinitesimal} (notation: $x\approx_\slow 0$) if $x\approx 0$ and $\frac1{\abs{x}}$ is slow scale, i.e., if
\[\caninf^a\le \abs{x}\le a,\quad \forall a\in\R_{>0}\]
and we call $x$ a \defstyle{fast scale infinitesimal} (notation: $x\approx_\fast 0$) if %$\frac1{\abs{x}}$ is fast scale, %now $\frac1{\abs{x}}$ may be ill-defined
\[\abs{x}\le \caninf^a,\quad \text{for some } a\in\R_{>0}.\]
The net $(x_\eps)_\eps$ corresponding to a fast scale (resp.\ slow scale) infinitesimal is said to tend slow scale (resp.\ fast scale) to $0$. A fast-scale infinitesimal (and the corresponding net) is also sometimes called `strongly associated with $0$'. We write $x\approx_\fast y$ or $x\in\monad_\fast(y)$ (resp.\ $x\approx_\slow y$ or $x\in\monad_\slow(y)$) for $x-y\approx_\fast 0$ (resp.\ $x-y\approx_\slow 0$).\\
We call a \defstyle{slow scale neighbourhood} of $x_0\in\GenR^d$ any set that contains $\{x\in\GenR^d: \abs{x-x_0}\le r\}$ for some $r\approx_\slow 0$ ($r\in\GenR_{>0}$).
A \defstyle{conic slow scale neighbourhood} of $\xi_0\in \widetilde\Sphere$ is a cone $\Gamma\subseteq\GenR^d$ with vertex $0$ that contains a slow scale neighbourhood of $\xi_0$ (i.e., there exists some $r\approx_\slow 0$ ($r\in\GenR_{>0}$) such that $\abs[\big]{\frac{\xi}{\abs\xi}-\xi_0}\le r\implies \xi \in \Gamma$). We write $\Gamma_\fast(\xi_0):=\{\xi\in \GenR^d: \frac{\xi}{\abs\xi}\approx_\fast \xi_0\}$.

In \cite{Hormann99,GarettoPhD}, $u=[u_\eps]\in\Gen(\Omega)$ is called $\Gen^\infty$-microlocally regular at $(x_0,\xi_0)\in\Omega\times \Sphere$ if there exist $\phi\in\test(\Omega)$ with $\phi(x_0)=1$, a (nongeneralized) conic neighbourhood $\Gamma$ of $\xi_0$ and $N\in\N$ such that for all $m\in\N$
\begin{equation}\label{eq-classical-regularity}
\sup_{\xi\in\Gamma}\langle\xi\rangle^m\abs[\big]{\fourier{\phi u_\eps}(\xi)}\le\eps^{-N},\quad\text{ for small $\eps$}.
\end{equation}

\section{$\widetilde\Gen^\infty$-microlocal regularity}
\begin{df}\label{df-regular}
$u\in \Gen(\Omega)$ is $\widetilde\Gen^\infty$-microlocally regular at $(x_0,\xi_0)\in\widetilde \Omega_c\times \widetilde\Sphere$ if there exists $v\in\Gen_c(\Omega)$ such that
\[u(x)=v(x), \ \forall x\approx_\fast x_0 \qquad \text{and}\qquad \fourier v(\xi)= 0,\ \forall \xi \in \widetilde\R^d_{fs} \cap \Gamma_\fast(\xi_0).\]
\end{df}

\begin{rem}
$v\in\Gen_c(\Omega)\subseteq \Gen_\Schwartz(\R^d)$, hence its Fourier transform $\fourier v\in\Gen_\Schwartz(\R^d)$. If $v=[v_\eps]$, then $\fourier v=[\fourier{v_\eps}]$ with $(\fourier{v_\eps})_\eps\in\EMod_\Schwartz$ if $(v_\eps)_\eps\in\EMod_\Schwartz$ (here, we denote $\Gen_\Schwartz(\R^d):=\EMod_\Schwartz/\Null_\Schwartz$) \cite[Def.\ 2.10]{GGO05}.
\end{rem}

We can equivalently reformulate the conditions in Def.\ \ref{df-regular} (some of the reformulations resemble the classical definitions more closely). The proof relies on the overspill principle \cite{HVInternal} (however, the proof below is self-contained).
\begin{lemma}\label{equiv-local-equality}
For $u=[u_\eps]$, $v=[v_\eps]\in \Gen(\Omega)$ and $x_0=[x_{0,\eps}]\in\widetilde \Omega_c$, the following are equivalent:
\begin{enumerate}
\item there exists a slow scale nbd.\ $V$ (in $\widetilde\Omega_c$) of $x_0$ such that $u(x)=v(x)$, $\forall x\in V$
\item $u(x)=v(x)$, $\forall x\approx_\fast x_0$
\item for each $m\in\N$
\[
\sup_{\abs{x-x_{0,\eps}}\le\eps^{1/m}} \abs{u_\eps(x)-v_\eps(x)}\le\eps^m,\quad\text{for small $\eps$}.
\]
\end{enumerate}
\end{lemma}
\begin{proof}
$(1)\implies(2)$: clear.

$(2)\implies (3)$: assuming that (3) does not hold, we find $m\in\N$ and a decreasing sequence $(\eps_n)_n$ tending to $0$ such that for each $n$ there exists $x_{\eps_n}\in\R^d$ with $\abs{x_{\eps_n}-x_{0,\eps_n}}\le\eps_n^{1/m}$ and $\abs{u_{\eps_n}(x)-v_{\eps_n}(x)}>\eps_n^m$. We can extend these $x_{\eps_n}$ to a net $(x_\eps)_\eps$ such that, with $x:=[x_\eps]$, $\abs{x-x_0}\le\caninf^{1/m}$, but $u(x)\ne v(x)$.

$(3)\implies(1)$: for any $m\in\N$, let the condition in (4) hold for $\eps\le\eps_m$. W.l.o.g., $(\eps_m)_m$ decreasingly tends to $0$. Let $m_\eps:= m$ for each $\eps\in (\eps_{m+1},\eps_m]$. Then $m_\eps\to+\infty$ als $\eps\to 0$, and for small $\eps$, we have that $\abs{u_\eps(x_\eps)-v_\eps(x_\eps)}\le\eps^{m_\eps}$ as soon as $\abs{x_\eps-x_{0,\eps}}\le\eps^{1/m_\eps}$. Hence $u(x)=v(x)$ for each $x\in V:=\{x\in\GenR^d: \abs{x-x_0}\le r\}$, with $r:=[\eps^{1/m_\eps}]\approx_\slow 0$.
\end{proof}

\begin{lemma}\label{equiv-zero-in-cone}
For $v=[v_\eps]\in \Gen_c(\Omega)$ and $\xi_0=[\xi_{0,\eps}]\in\widetilde\Sphere$, the following are equivalent:
\begin{enumerate}
\item there exists a conic slow scale nbd.\ $\Gamma$ (in $\GenR^d$) of $\xi_0$ and $r\in\GenR_{ss}$ s.t.\ $\fourier{v}(\xi)=0$, $\forall\xi\in \Gamma$ with $\abs{\xi}\ge r$
\item $\fourier v(\xi)= 0,\ \forall \xi \in \widetilde\R^d_{fs}\cap\Gamma_\fast(\xi_0)$
\item (assuming $(v_\eps)_\eps\in \EMod_{\Schwartz}$) for each $m\in\N$,
\[
\sup_{\abs{\frac\xi{\abs\xi}-\xi_{0,\eps}}\le\eps^{1/m},\, \abs\xi\ge\eps^{-1/m}}\abs{\fourier v_\eps(\xi)}\le\eps^m,\quad\text{for small $\eps$}
\]
\item (assuming $(v_\eps)_\eps\in \EMod_{\Schwartz}$) there exists $N\in\N$ such that for each $m\in\N$, 
\[\sup_{\abs{\frac\xi{\abs\xi}-\xi_{0,\eps}}\le\eps^{1/m}}\langle\xi\rangle^{m}\abs{\fourier v_\eps(\xi)}\le\eps^{-N},\quad\text{for small $\eps$}.
\]
\end{enumerate}
\end{lemma}
\begin{proof}
$(1)\implies(2)\implies(3)\implies(1)$, $(2)\implies(4)$: similar to Lemma \ref{equiv-local-equality}. Since $(\fourier{v_\eps})_\eps\in\EMod_\Schwartz$, we find for each $n,m\in\N$ some $N\in\N$ such that
\[\sup_{\abs\xi\ge \eps^{-N}}\langle\xi\rangle^m\abs{\fourier {v_\eps}(\xi)}\le\eps^n.\] This ensures that the nets $(\xi_\eps)_\eps$ constructed in the $(2)\Rightarrow(3)$ and $(2)\Rightarrow(4)$ parts of the proof are moderate.

$(4)\implies (3)$: straightforward (cf.\ also \cite[Thm.~3.12]{HVPointwise}).
\end{proof}

We denote $\Gen_{ss}(\GenR^d)= \{u\in\Gen_\Schwartz(\R^d): u(x)=0$ for each $x\in \GenR^d_{fs}\}$.

\begin{prop}\label{cut-off}
Let $v\in\Gen_c(\Omega)$, $\phi\in\Gen^\infty(\Omega)$ and $\xi_0\in\widetilde\Sphere$. Let $\fourier v(\xi)=0$ for each $\xi\in\GenR^d_{fs}\cap\Gamma_\fast(\xi_0)$. Then also $\fourier{\phi v}(\xi)=0$ for each $\xi\in\GenR^d_{fs}\cap\Gamma_\fast(\xi_0)$.
\end{prop}
\begin{proof}
W.l.o.g., $\phi\in\Gen^\infty_c(\Omega)$.
Fix $\xi\in \GenR^d_{fs}\cap\Gamma_\fast(\xi_0)$. Then
\[\fourier{\phi v}(\xi) = \int\fourier\phi(\eta)\fourier v(\xi-\eta)\,d\eta.\]
Since $\Gen_{ss}(\GenR^d)\cap \fourier{\Gen_{ss}}(\GenR^d) = \Gen^\infty_\Schwartz(\GenR^d)$ \cite[Thm.~3.13]{HVPointwise}, we have $\fourier\phi\in\fourier{\Gen_c^\infty(\Omega)}\subseteq \fourier{\Gen^\infty_\Schwartz}\subseteq \Gen_{ss}(\GenR^d)$.

Now let $\eta\in\GenR^d_{ss}$. Then also $\xi-\eta\in\GenR^d_{fs}$. Since $\frac{\abs\eta}{\abs\xi}\approx_\fast 0$,
\[
\frac{\xi-\eta}{\abs{\xi-\eta}}\approx_\fast \frac{\xi-\eta}{\abs\xi}\approx_\fast \frac{\xi}{\abs\xi}\approx_\fast \xi_0.
\]
Hence $\fourier v(\xi-\eta)=0$.

Therefore, $\eta\mapsto \fourier\phi(\eta)\fourier v(\xi-\eta)=0$ in $\Gen_\Schwartz$ \cite[Lemma~4.1]{HVPointwise}. Thus $\fourier{\phi v}(\xi)=0$.
\end{proof}

\begin{cor}\label{regularity-via-cut-off}
Let $\phi\in\Gen^\infty(\Omega)$. If $u\in\Gen(\Omega)$ is $\widetilde\Gen^\infty$-microlocally regular at $(x_0,\xi_0)$, then also $\phi u$ is $\widetilde\Gen^\infty$-microlocally regular at $(x_0,\xi_0)$.
\end{cor}

Let $B(a,r):=\{x\in \R^d: \abs{x-a}<r\}$. We use the following notation. Let a net $(x_{0,\eps})_\eps$ be given. We fix $\phi_0\in\test(B(0,1))$ with $0\le\phi_0\le 1$ and with $\phi_0(x)=1$ for each $x\in B(0,1/2)$. For $m\in\N$ and $\eps>0$, we denote
\[\phi_{m,\eps}(x):= \phi_0\Bigl(\frac{x-x_{0,\eps}}{\eps^{1/m}}\Bigr).\]

\begin{cor}\label{def-regular-with-cutoff}
For $u\in \Gen(\Omega)$ and $(x_0,\xi_0)\in\widetilde \Omega_c\times \widetilde\Sphere$, the following are equivalent:
\begin{enumerate}
\item $u$ is $\widetilde\Gen^\infty$-microlocally regular at $(x_0,\xi_0)$
\item there exists $\phi\in\Gen_c^\infty(\Omega)$ such that
\[\phi(x)=1, \ \forall x\approx_\fast x_0 \qquad \text{and}\qquad \fourier{\phi u}(\xi)= 0,\ \forall \xi \in \widetilde\R^d_{fs}\cap\Gamma_\fast(\xi_0)\]
\item there exists $\phi\in\Gen_c^\infty(\Omega)$ and $R\in\GenR_{ss}$ such that
\[\begin{cases}
\abs{\partial^\alpha\phi(x)}\le R\\
\abs{\phi(x)}\ge \frac1R,
\end{cases}
\!\forall x\approx_\fast x_0, \ \forall \alpha\in\N^d \text{ and }\ \fourier{\phi u}(\xi)= 0,\ \forall \xi \in \widetilde\R^d_{fs}\cap\Gamma_\fast(\xi_0).\]
\end{enumerate}
\end{cor}
\begin{proof}
$(1)\Rightarrow(2)$: choose $v$ as in the definition of $\widetilde\Gen^\infty$-microlocal regularity. By characterization (1) in Lemma \ref{equiv-local-equality}, we can find $\phi\in\Gen_c^\infty(\Omega)$ with $\phi(x)=1$ for each $x\approx_\fast x_0$ and with $\phi u=\phi v$. By Proposition \ref{cut-off}, $\fourier{\phi u}(\xi)=\fourier{\phi v}(\xi)=0$ for each $\xi\in\GenR^d_{fs}\cap\Gamma_\fast(\xi_0)$.

$(2)\implies(3)$: trivial.

$(3)\implies(1)$: by an overspill argument, actually $\abs{\partial^\alpha\phi(x)}\le R$ and $\abs{\phi(x)}\ge 1/R$ hold on some slow scale neighbourhood $V=\{x\in \GenR^d: \abs x\le [\eps^{1/m_\eps}]\}$ (with $m_\eps\to+\infty$). Then $\psi:=[\phi_{m_\eps,\eps}]\in\Gen_c^\infty(\Omega)$ with $\psi(x)=1$ for each $x\approx_\fast x_0$ and $\frac\psi\phi\in\Gen^\infty_c(\Omega)$, whence $\fourier{\psi u}(\xi)=\fourier{\frac{\psi}{\phi}\phi u}(\xi)=0$ for each $\xi \in \widetilde\R^d_{fs}\cap\Gamma_\fast(\xi_0)$ by Proposition \ref{cut-off}. Then let $v:=\psi u\in \Gen_c(\Omega)$.
\end{proof}

\section{Consistency with $\widetilde\Gen^\infty$-regularity}
We now proceed to show that the projection of the wave front set in the first coordinate is the singular support (Theorem \ref{proj-of-WF-is-singular-support}).
\begin{lemma}\label{lemma-projection-of-wave-front-1}
Let $x_0=[x_{0,\eps}]\in\widetilde\Omega_c$ and $\Gamma\subseteq \R^d$ a (nongeneralized) cone. Let $u=[u_\eps]\in\Gen(\Omega)$ be $\widetilde\Gen^\infty$-microlocally regular at $(x_0,\xi_0)$, for each $\xi_0\in\widetilde\Sphere\cap\widetilde\Gamma$. Then for each $m\in\N$
\begin{equation}\label{eq-regular-for-each-xi-1}
(\exists \eps_m) (\forall \eps\le\eps_m) (\forall \xi_0\in\Sphere\cap \Gamma) (\exists k\in\N, k\ge 2m) \Bigl(\sup_{\abs{\frac\xi{\abs\xi}-\xi_0}\le\eps^{1/m},\, \abs\xi\ge\eps^{-1/m}}\abs[\big]{\fourier{\phi_{k,\eps} u_\eps}(\xi)}\le\eps^m\Bigr).
\end{equation}
\end{lemma}
\begin{proof}
Let $m\in\N$. If \eqref{eq-regular-for-each-xi-1} does not hold, then we find a decreasing sequence $(\eps_n)_n$ tending to $0$ and $\xi_{0,\eps_n}\in\Sphere\cap\Gamma$ such that for each $k\ge 2m$
\[\sup_{\abs{\frac\xi{\abs\xi}-\xi_{0,\eps_n}}\le\eps^{1/m},\, \abs\xi\ge\eps^{-1/m}}\abs[]{\fourier{\phi_{k,\eps_n} u_{\eps_n}}(\xi)}>\eps_n^m.\]
Extend $\xi_{0,\eps_n}$ to a net $(\xi_{0,\eps})_\eps$ representing $\xi_0\in\widetilde\Sphere\cap\widetilde\Gamma$. As $u$ is $\widetilde\Gen^\infty$-microlocally regular at $(x_0,\xi_0)$, there exists $v\in\Gen_c(\Omega)$ such that $u=v$ on a slow scale neighbourhood $V$ of $x_0$ and $\fourier{v}(\xi)=0$ for each $\xi\in\GenR^d_{fs}$ in a conic slow scale neighbourhood of $\xi_0$. As in Corollary \ref{def-regular-with-cutoff}, this also holds if we replace $v$ by $\phi u$, where $\phi= [\phi_{k_\eps,\eps}]$, provided that $k_\eps\to \infty$ as $\eps\to 0$ (ensuring that $\phi\in\Gen^\infty(\R^d)$ and $\phi=1$ on a slow scale neighbourhood of $x_0$) and provided that $\{x\in\GenR^d: \abs{x-x_0}\le [\eps^{1/k_\eps}]\}\subseteq V$ (ensuring that $\phi u = \phi v$). W.l.o.g., $k_\eps\ge 2 m$ for each $\eps$. 

Thus for each $n$, we find $\xi_{\eps_n}$ such that $\abs[\big]{\frac{\xi_{\eps_n}}{\abs{\xi_{\eps_n}}}-\xi_{0,\eps_n}}\le\eps_n^{1/m},\, \abs{\xi_{\eps_n}}\ge\eps_n^{-1/m}$ and $\abs[]{\fourier{\phi_{k_{\eps_n},\eps_n} u_{\eps_n}}(\xi_{\eps_n})}>\eps_n^m$. Extending $\xi_{\eps_n}$ to a net $(\xi_\eps)_\eps$ representing $\xi\in\GenR^d_{fs}$ with $\abs[\big]{\frac\xi{\abs\xi}-\xi_{0}}\le\caninf^{1/m}$ (moderateness of $(\xi_\eps)_\eps$ necessarily follows from $(\fourier{\phi_{k_\eps,\eps}u_\eps})_\eps\in\EMod_\Schwartz$), we have ${\fourier{\phi u}(\xi)}\ne 0$, a contradiction.
\end{proof}

The following lemma is an easy generalization of \cite[Lemma 8.1.1]{Hormander}:
\begin{lemma}\label{lemma-cone}
Let $\xi_0\in\widetilde\Sphere$, $r\in\GenR_{>0}$, $r\le 1/2$ and $\Gamma = \{\xi\in\GenR^d: \abs[\big]{\frac{\xi}{\abs\xi}-\xi_0}\le 3r\}$. Then for each $\zeta\in \Gamma' = \{\xi\in\GenR^d: \abs[\big]{\frac{\xi}{\abs\xi}-\xi_0}\le r\}$ and for each $\eta\in \GenR^d$ with $\abs{\zeta-\eta}\le r\abs\zeta$, we have $\eta\in \Gamma$.
\end{lemma}
\begin{proof}
Since $\abs{\eta}\ge \abs{\zeta}-\abs{\zeta-\eta}\ge \frac 12\abs{\zeta}$, also $\abs{\eta}$ is invertible. Then
\[\abs{\frac\eta{\abs\eta}-\xi_0}\le \abs{\frac\eta{\abs\eta}-\frac\eta{\abs\zeta}} + \abs{\frac\eta{\abs\zeta}-\frac\zeta{\abs\zeta}} +\abs{\frac\zeta{\abs\zeta}-\xi_0}\le 3 r.\]
\end{proof}

\begin{lemma}\label{lemma-projection-of-wave-front-2}
Let $x_0=[x_{0,\eps}]\in\widetilde\Omega_c$ and $\Gamma\subseteq \R^d$ a (nongeneralized) cone. Let $u=[u_\eps]\in\Gen(\Omega)$ be $\widetilde\Gen^\infty$-microlocally regular at $(x_0,\xi_0)$, for each $\xi_0\in\widetilde\Sphere\cap\widetilde\Gamma$. Then for each $m\in\N$
\begin{equation}\label{eq-regular-for-each-xi-2}
(\exists \eps_m) (\forall \eps\le\eps_m) (\forall \zeta\in\Gamma, \abs{\zeta}\ge \eps^{-2/m}) \bigl(\abs[\big]{\fourier{\phi_{2m,\eps} u_\eps}(\zeta)}\le\eps^{m-1}\bigr).
\end{equation}
\end{lemma}
\begin{proof}
Let $m\in\N$ and $\eps$ sufficiently small (such that \eqref{eq-regular-for-each-xi-1} holds). Let $\zeta\in\Gamma$ with $\abs\zeta\ge \eps^{-2/m}$. Then $\xi_0:=\zeta/\abs{\zeta}\in\Sphere\cap\Gamma$. Denote $v_{k,\eps}:=\phi_{k,\eps}u_\eps$. Then there exists $k\ge 2m$ such that $\abs[\big]{\fourier{v_{k,\eps}}(\xi)}\le\eps^m$ for each $\xi\in\R^d$ with $\abs{\xi}\ge\eps^{-1/m}$ and $\abs[\big]{\frac{\xi}{\abs\xi}-\xi_0}\le\eps^{1/m}$. If $k>2m$, then we have $\phi_{2m,\eps}=\phi_{2m,\eps}\phi_{k,\eps}$ as soon as $\eps$ is sufficiently small (only depending on $m$, e.g.\ as soon as $\frac12 \eps^{1/(2m+1)}\ge \eps^{1/(2m)}$). Hence also $v_{2m,\eps}=\phi_{2m,\eps}v_{k,\eps}$, and
\[\fourier {v_{2m,\eps}} (\zeta)= \int_{\R^d}\fourier{\phi_{2m,\eps}}(\zeta-\eta)\fourier {v_{k,\eps}}(\eta)\,d\eta.
\]
Let $A:=\{\eta\in\R^d: \abs{\zeta-\eta}\le\frac{\eps^{1/m}\abs{\zeta}}3\}$. By Lemma \ref{lemma-cone}, $\abs[\big]{\fourier{v_{k,\eps}}(\eta)}\le \eps^m$ for each $\eta\in A$, hence
\[\abs{\int_A\fourier{\phi_{2m,\eps}}(\zeta-\eta)\fourier {v_{k,\eps}}(\eta)\,d\eta}\le\eps^m \norm[\big]{\fourier{\phi_{2m,\eps}}}_{L^1(\R^d)}= \eps^m\norm[\big]{\fourier{\phi_{0}}}_{L^1(\R^d)} =:C\eps^m.\]
Further, $\fourier{u}\in\Gen_\Schwartz(\R^d)$, hence $\norm{\fourier{u_\eps}}_{L^1}\le\eps^{-N}$ for some $N\in\N$. Also $\fourier{\phi_0}\in\Schwartz(\R^d)$, hence $\abs[\big]{\fourier{\phi_0}(\xi)}\le C_m \langle\xi\rangle^{-2m(m+N)}$ for each $\xi\in\R^d$. Thus for $\eta\in \R^d\setminus A$, $\abs{\zeta-\eta}\ge \eps^{1/m}\abs{\zeta}/3\ge \eps^{-1/m}/3$ and
\[\abs[\big]{\fourier{\phi_{2m,\eps}}(\zeta-\eta)}=\eps^{d/(2m)}\abs[\big]{\fourier{\phi_0}(\eps^{1/(2m)}(\zeta-\eta))} \le C_m (\eps^{-1/(2m)}/3)^{-2m(m+N)}= C'_m \eps^{m+N},\]
and thus
\[
\abs{\int_{\R^d\setminus A}\fourier{\phi_{2m,\eps}}(\zeta-\eta)\fourier {v_{k,\eps}}(\eta)\,d\eta}\le C'_m \eps^{m+N}\norm{\fourier{v_{k,\eps}}}_{L^1}\le C'_m\eps^{m+N}\norm[\big]{\fourier{\phi_{k,\eps}}}_{L^1}\norm{\fourier{u_\eps}}_{L^1}\le C''_m\eps^m,
\]
since $\norm[\big]{\fourier{\phi_{k,\eps}}}_{L^1}=\norm[\big]{\fourier{\phi_{0}}}_{L^1}$.
\end{proof}

In \cite{Dapic98}, several notions of pointwise regularity for elements of $\Gen(\Omega)$ were considered. By \cite[Cor.~5.5]{HVPointwise}, it follows that one of them is more fundamental, in the sense that the other notions can be described in terms of it. We therefore restrict to this notion only, slightly simplifying the notations (it was called $\Gentdinfty$-regularity in \cite{Dapic98,HVPointwise}):
\begin{df}
Let $u\in\Gen(\Omega)$. Let $x\in\widetilde\Omega_c$. Then $u\in\widetilde\Gen^\infty(x)$ if there exists $N\in\N$ (independent of $\alpha$) such that
\[\abs{\partial^\alpha u(x)}\le \caninf^{-N},\quad\forall \alpha\in\N^d.\]
Let $A\subseteq \GenR^d$. Then $u\in\widetilde\Gen^\infty(A)$ if $u\in\widetilde\Gen^\infty(x)$ for each $x\in A$.
\end{df}

\begin{thm}\label{proj-of-WF-is-singular-support}
Let $x_0\in\widetilde\Omega_c$. For $u\in\Gen(\Omega)$, the following are equivalent:
\begin{enumerate}
\item $u$ is $\widetilde\Gen^\infty$-microlocally regular at $(x_0,\xi_0)$, for each $\xi_0\in\widetilde\Sphere$
\item $u\in\widetilde\Gen^\infty(V)$ for some slow scale neighbourhood $V$ of $x_0$.
\end{enumerate}
\end{thm}
\begin{proof}
$(1)\Rightarrow(2)$: Let $\eps_m$ be as in Lemma \ref{lemma-projection-of-wave-front-2} with $\Gamma=\R^d$. W.l.o.g., $(\eps_m)_m$ decreasingly tends to $0$. Let $m_\eps:= m$ for each $\eps\in (\eps_{m+1},\eps_m]$. Then $m_\eps\to+\infty$ als $\eps\to 0$, and for small $\eps$, we have that $\abs[\big]{\fourier{\phi_{m_\eps,\eps} u_\eps}(\zeta)}\le\eps^{m_\eps-N-1}$ as soon as $\abs{\zeta}\ge \eps^{-2/m_\eps}$. We conclude that $\fourier{\phi u}(\xi)=0$ for each $\xi\in\GenR^d_{fs}$ and for $\phi:=[\phi_{m_\eps,\eps}]\in\Gen_c^\infty(\R^d)$. Hence $\phi u\in \fourier{\Gen_{ss}}(\R^d)\cap \Gen_c(\Omega)\subseteq\Gen_c^\infty(\Omega)$. Since $\phi=1$ on a slow scale neighbourhood $V$ of $x_0$, we have that $u\in \widetilde\Gen^\infty(V)$.

$(2)\Rightarrow(1)$: there exists $\phi\in\Gen_c^\infty(\Omega)$ with $\phi = 1$ on a slow scale neighbourhood of $x_0$ and with $\phi u\in\Gen_c^\infty(\R^d)$ (e.g., $\phi=[\phi_{m_\eps,\eps}]$ for a suitable $m_\eps\to \infty$). Since  $\phi u\in \Gen^\infty_\Schwartz(\R^d)\subseteq \fourier{\Gen_{ss}}(\R^d)$, $\fourier{\phi u}(\xi) = 0$ for each $\xi\in\GenR^d_{fs}$.
\end{proof}

Again, we can equivalently reformulate the condition in the previous theorem:
\begin{prop}
For $u\in\Gen(\Omega)$, the following are equivalent:
\begin{enumerate}
\item $u\in\widetilde\Gen^\infty(V)$ for some slow scale neighbourhood of $V$ of $x_0$
\item $(\exists N\in\N)$ $(\forall \alpha\in\N^d)$ $(\forall x\in\GenR^d, x\approx_\fast x_0)$ $(\abs{\partial^\alpha u(x)}\le\caninf^{-N})$
\item $(\exists N\in\N)$ $(\forall m\in\N)$ $(\forall \alpha\in\N^d)$
\[
\sup_{\abs{x-x_{0,\eps}}\le\eps^{1/m}}\abs{\partial^\alpha u_\eps(x)}\le\eps^{-N},\quad\text{for small $\eps$}.
\]
\end{enumerate}
\end{prop}
\begin{proof}
$(1)\implies (2)$: there exist $n_\eps\to \infty$ (as $\eps\to 0$) such that the internal set \cite{HVInternal} $\{x\in\GenR^d: \abs{x-x_0}\le[\eps^{1/n_\eps}]\} = [\{x\in\R^d: \abs{x-x_{0,\eps}}\le \eps^{1/n_\eps}\}]\subseteq V$. The result follows by \cite[Prop.\ 5.3]{HVPointwise}.

$(2)\implies (3)$: let $N\in\N$ as in (2). Assuming that (3) does not hold, we find some $m\in\N$, some $\alpha\in\N^d$, some decreasing sequence $(\eps_n)_n$ tending to $0$ and some $x_{\eps_n}\in\R^d$ with $\abs{x_{\eps_n}-x_{0,{\eps_n}}}\le\eps_n^{1/m}$ and $\abs{\partial^\alpha u_{\eps_n}(x_{\eps_n})}>\eps_n^{-N-1}$, for each $n$. Extending the sequence $(x_{\eps_n})_n$ to a net $(x_\eps)_\eps$, we obtain $x=[x_\eps]\in\GenR^d$ with $\abs{x-x_0}\le\caninf^{1/m}$ and $\abs{\partial^\alpha u(x)}\not\le\caninf^{-N}$, contradicting (2).

$(3)\implies (1)$: for each $m\in\N$, there exists $\eps_m>0$ such that
\[
\sup_{\abs{\alpha}\le m,\,\abs{x-x_{0,\eps}}\le\eps^{1/m}}\abs{\partial^\alpha u_\eps(x)}\le\eps^{-N},\quad \forall \eps\le\eps_m. 
\]
W.l.o.g., $(\eps_m)_m$ decreasingly tends to $0$. Let $m_\eps:= m$ for each $\eps\in (\eps_{m+1},\eps_m]$. Then $m_\eps\to+\infty$ als $\eps\to 0$, and for small $\eps$, we have that $\sup_{\abs{\alpha}\le m_\eps,\,\abs{x-x_{0,\eps}}\le\eps^{1/m_\eps}}\abs{\partial^\alpha u_\eps(x)}\le\eps^{-N}$. Then $u\in\widetilde\Gen^\infty(V)$ for $V:=\{x\in\GenR^d: \abs{x-x_0}\le [\eps^{1/m_\eps}]\}$.
\end{proof}
This should be distinguished from $u\in\widetilde\Gen^\infty(\monad_\fast(x_0))$, which is (again by \cite[Prop.\ 5.3]{HVPointwise}) equivalent with $(\forall m\in\N)$ $(\exists N\in\N)$ $(\forall \alpha\in\N^d)$
\[\sup_{\abs{x-x_{0,\eps}}\le\eps^{1/m}}\abs{\partial^\alpha u_\eps(x)}\le\eps^{-N},\quad\text{for small $\eps$}.\]

\section{Consistency with $\Gen^\infty$-microlocal regularity}
We first give some pointwise characterizations of $\Gen^\infty$-microlocal regularity:
\begin{lemma}\label{equiv-zero-in-standard-cone}
For $v=[v_\eps]\in \Gen_c(\Omega)$ and $\Gamma\subseteq\R^d$ a (nongeneralized) cone, the following are equivalent:
\begin{enumerate}
\item $\fourier v(\xi)= 0,\ \forall \xi \in \widetilde\Gamma\cap \widetilde\R^d_{fs}$
item (assuming $(v_\eps)_\eps\in \EMod_{\Schwartz}$) for each $m\in\N$,
\[
\sup_{\xi\in\Gamma,\, \abs\xi\ge\eps^{-1/m}}\abs{\fourier v_\eps(\xi)}\le\eps^m,\quad\text{for small $\eps$}
\]
\item (assuming $(v_\eps)_\eps\in \EMod_{\Schwartz}$) there exists $N\in\N$ such that for each $m\in\N$, 
\[\sup_{\xi\in\Gamma}\langle\xi\rangle^{m}\abs{\fourier v_\eps(\xi)}\le\eps^{-N},\quad\text{for small $\eps$}.
\]
\end{enumerate}
\end{lemma}
\begin{proof}
Similar to Lemma \ref{equiv-zero-in-cone} (cf.\ also \cite[Thm.~3.12]{HVPointwise}).
\end{proof}

\begin{lemma}\label{classical-regularity-expressed-pointwise}
Let $u\in\Gen(\Omega)$. Then the following are equivalent:
\begin{enumerate}
\item $u$ is $\Gen^\infty$-microlocally regular at $(x_0,\xi_0)\in\Omega\times \Sphere$
\item there exist $\phi\in\test(\Omega)$ with $\phi(x_0)=1$ and a conic neighbourhood $\Gamma$ of $\xi_0$ s.t.
\[
\fourier{\phi u}(\xi)=0,\quad \forall \xi\in\widetilde\Gamma\cap\GenR^d_{fs}.
\]
\item there exist $\phi\in\Gen_c^\infty(\Omega)$, a (nongeneralized) neighbourhood $V$ of $x_0$, a conic neighbourhood $\Gamma$ of $\xi_0$ and $R\in\GenR_{ss}$ s.t.
\[\begin{cases}
\abs{\partial^\alpha\phi(x)}\le R\\
\abs{\phi(x)}\ge \frac1R,
\end{cases}
\!\forall x\in\widetilde V, \ \forall \alpha\in\N^d \text{ and }\ \fourier{\phi u}(\xi)= 0,\ \forall \xi \in \widetilde\Gamma\cap \widetilde\R^d_{fs}.\]
\end{enumerate}
\end{lemma}
\begin{proof}
$(1)\Leftrightarrow(2)$: by Lemma \ref{equiv-zero-in-standard-cone}.

$(2)\implies(3)$: any $\phi\in\test(\Omega)$ with $\phi(x_0)=1$ satisfies the requirements of $\phi$ in (3).

$(3)\implies(2)$: choose $\psi\in\test(V)$ with $\psi(x_0)=1$. As in Corollary \ref{def-regular-with-cutoff}, $\frac\psi\phi\in\Gen^\infty_c(\Omega)$ and we can find a conic neighbourhood $\Gamma'\subset\Gamma$ of $x_0$ s.t.\ $\fourier{\psi u}(\xi)=0$ for each $\xi\in\widetilde{\Gamma'}\cap\GenR^d_{fs}$ by Proposition \ref{cut-off}.
\end{proof}

\begin{thm}\label{consistency}
Let $u\in\Gen(\Omega)$. Then $u$ is $\Gen^\infty$-microlocally regular at $(x_0,\xi_0)\in\Omega\times \Sphere$ iff there exists $r\in\R_{>0}$ such that $u$ is $\widetilde\Gen^\infty$-microlocally regular at $(x,\xi)$ for each $x\in\GenR^d$ with $\abs{x-x_0}\le r$ and each $\xi\in\widetilde \Sphere$ with $\abs{\xi-\xi_0}\le r$.
\end{thm}
\begin{proof}
$\Rightarrow$: by Corollary \ref{def-regular-with-cutoff} and Lemma \ref{classical-regularity-expressed-pointwise}.

$\Leftarrow$: Let $V:= B(0,r)$ and $\Gamma:=\{\xi\in\R^d\setminus\{0\}: \abs[\big]{\frac{\xi}{\abs\xi}-\xi_0}<r\}$. Let $m\in\N$. Let as before $\phi_{m,\eps}(x):=\phi_0\bigl(\frac{x-x_0}{\eps^{1/m}}\bigr)$ (now with the given $x_0\in\Omega$). Denote the translation $\tau_y(x):= y+x$. Then \eqref{eq-regular-for-each-xi-2} holds even with $\tau_x (\phi_{2m,\eps})$ instead of $\phi_{2m,\eps}$, for each $x\in\widetilde V$. This implies that
\[
(\exists \eps_m) (\forall \eps\le\eps_m) (\forall x\in V) (\forall \xi\in\Gamma, \abs{\xi}\ge \eps^{-2/m}) \bigl(\abs[\big]{\Fourier[\tau_x(\phi_{2m,\eps}) u_\eps](\xi)}\le\eps^{m-1}\bigr),
\]
since otherwise, one constructs $x=[x_\eps]\in \widetilde V$ contradicting \eqref{eq-regular-for-each-xi-2} for this particular $x$. W.l.o.g., $(\eps_m)_m$ decreasingly tends to $0$. Let $m_\eps:= m$ for each $\eps\in (\eps_{m+1},\eps_m]$. Then $m_\eps\to+\infty$ als $\eps\to 0$, and for small $\eps$, we have that
\[(\forall x\in V) (\forall \xi\in\Gamma, \abs{\xi}\ge \eps^{-2/m_\eps}) \bigl(\abs[\big]{\Fourier[\tau_x(\phi_{2m_\eps,\eps}) u_\eps](\xi)}\le\eps^{m_\eps-1}\bigr).\]
Now consider a grid
\[G_\eps:=\Bigl\{\tfrac{\eps^{1/(2m_\eps)}}{\sqrt d}(k_1,\dots,k_d): k_j\in\Z, \abs{k_j} < r\eps^{-1/(2m_\eps)}\ (j=1,\dots,d)\Bigr\}.\]
Then $G_\eps$ contains at most $(2r\eps^{-1/(2m_\eps)}+1)^d\le 2 (2r)^d\eps^{-d/(2m_\eps)}$ elements, for small $\eps$. Let $\psi_\eps:= \sum_{y\in G_\eps}\tau_y(\phi_{2m_\eps,\eps})$. As $G_\eps\subseteq V$, $\abs[\big]{\fourier{\psi_\eps u_\eps}(\xi)}\le 2(2r)^d\eps^{m_\eps-1-d/(2m_\eps)}$ for each $\xi\in\Gamma$ with $\abs{\xi}\ge \eps^{-2/m_\eps}$, for small $\eps$.

Now let $x\in W:=\{x\in\R^d:\max\{\abs{x_1},\dots, \abs{x_d}\}<r/\sqrt d\}$ arbitrary. Then there exists $y\in G_\eps$ such that $\abs{x_j-y_j}< \tfrac{\eps^{1/(2m_\eps)}}{2\sqrt d}$ ($j=1,\dots,d$), hence $\psi_\eps(x_0+x)\ge \tau_{-y}(\phi_{2m_\eps,\eps})(x_0+x) = 1$. On the other hand, there is only at most a finite number $C_d$ (only depending on $d$) of elements $y\in G_\eps$ for which $\abs{x-y}\le \eps^{1/(2m_\eps)}$. Hence for each $\alpha\in\N^d$,
\[\abs{\partial^\alpha \psi_\eps(x_0+x)}\le C_d \sup_{y\in V} \abs{\partial^\alpha \phi_{2m_\eps,\eps}(y)}\le C_{d,\alpha} \eps^{-\abs{\alpha}/(2m_\eps)} \]
Hence $\psi:=[\psi_\eps]\in\Gen_c^\infty(\Omega)$, $\fourier{\psi u}(\xi)=0$ for each $\xi\in \GenR^d_{fs}\cap \widetilde\Gamma$, and there exists $c\approx_\slow 0$ such that $\abs{\partial^\alpha\psi(x)}\le 1/c$ and $\abs{\psi(x)}\ge c$ for each $x\in(\tau_{x_0}(W))\sptilde$ and $\alpha\in\N^d$. The result follows by Lemma \ref{classical-regularity-expressed-pointwise}.
\end{proof}

We denote by $\iota(u)\in\Gen(\Omega)$ the embedded image of $u\in\test'(\Omega)$.
\begin{cor}
Let $u\in\test'(\Omega)$ and $(x_0,\xi_0)\in \Omega\times\Sphere$. Then $(x_0,\xi_0)\in WF(u)$ iff for each $r\in\R_{>0}$, there exist $x\in\GenR^d$ with $\abs{x- x_0}\le r$ and $\xi\in\widetilde\Sphere$ with $\abs{\xi-\xi_0}\le r$ such that $\iota(u)$ is not $\widetilde\Gen^\infty$-microlocally regular at $(x,\xi)$.
\end{cor}
\begin{proof}
By \cite{GarettoPhD}, $(x_0,\xi_0)\in WF(u)$ iff $\iota(u)$ is not $\Gen^\infty$-microlocally regular at $(x_0,\xi_0)$. The result follows by Proposition \ref{consistency}.
\end{proof}

\end{document}